\documentclass[10pt]{amsart}

\usepackage{amsmath,comment}
\usepackage{amssymb}
\usepackage{enumerate}
\usepackage{amsthm}
\usepackage{enumitem}
\usepackage{hyperref}

\newcommand{\Fc}{\mathcal{F}}

\newcommand{\Zb}{\mathbb{Z}}
\newcommand{\Nb}{\mathbb{N}}

\newtheorem{theorem}{Theorem}[section]
\newtheorem{proposition}[theorem]{Proposition}

\newtheorem{corollary}[theorem]{Corollary}
\newtheorem{definition}[theorem]{Definition}
\newtheorem{lemma}[theorem]{Lemma}
\newtheorem{remark}[theorem]{Remark}
\newtheorem{example}[theorem]{Example}

\begin{document}

\title{Distributional chaos for composition operators on $L^{p}$-spaces}

%\date{\today}
\author[Shengnan He]{Shengnan He}
\address{School of Humanities and Fundamental Sciences, Shenzhen Institute of Information Technology, Shenzhen 518172, China.}
\email{heshnmath@sziit.edu.cn}

\author[Zongbin Yin]{Zongbin Yin}
\address{School of Mathematics and Systems Science, Guangdong Polytechnic Normal University, Guangzhou 510665, China.}
\email{yinzb\_math@163.com}
\begin{abstract}
In this paper, we investigate the distributional chaos of the composition operator \(T_{\varphi}:f\mapsto f\circ\varphi\) on \(L^{p}(X,\mathcal{B},\mu)\), \(1\leq p <\infty\). We provide a characterization and practical sufficient conditions on \(\varphi\) for \(T_{\varphi}\) to be distributionally chaotic. Furthermore, we show that the existence of a dense set of distributionally irregular vectors implies the existence of a dense distributionally chaotic set, without any additional condition. We also provide a useful criterion for densely distributional chaos. Moreover, we characterize the weight sequences that ensure distributional chaos for bilateral backward shifts, unilateral backward shifts, bilateral forward shifts, and unilateral forward shifts on the weighted \(\ell^{p}\)-spaces \(\ell^{p}(\Nb,v)\) and \(\ell^{p}(\Zb,v)\). As a consequence, we reveal the equivalence between distributional chaos and densely distributional chaos for backward shifts and forward shifts on \(\ell^{p}(\mathbb{Z},v)\) without any additional condition. Finally, we characterize the composition operator \(T_{\varphi}\) on \(L^{p}(\mathbb{T},\mathcal{B},\lambda)\) induced by an automorphism \(\varphi\) of the unit disk \(\mathbb{D}\). We show that \(T_{\varphi}\) is densely distributionally chaotic if and only if \(\varphi\) has no fixed point in \(\mathbb{D}\).

\end{abstract}
\maketitle

\section{Introduction}
During the last three decades, various chaotic behaviors of linear operators acting on infinite-dimensional Banach spaces have been observed and widely studied, including hypercyclicity, Devaney chaos, Li-Yorke chaos, and distributional chaos. For a good account of results concerning linear chaos, we refer the reader to \cite{Bayart, Bernardes2013, Bernardes2015, Bernardes2018, Bermudez2011, Peris} and the references therein.

Among the various properties related to chaos in the study of linear chaotic operators, distributional chaos plays a particularly significant role. The concept of distributional chaos was proposed by Schweizer and Smítal \cite{Schweizer1994} in 1994 to unify various concepts of chaos in interval maps. The earliest research on distributional chaos in linear systems appeared in 2006, when Oprocha \cite{Oprocha2006} introduced continuous linear operators exhibiting distributional chaos in infinite-dimensional spaces. In 2009, Martínez-Giménez et al. \cite{Martinez2009} proved that Devaney chaotic operators on Köthe sequence spaces are distributionally chaotic, and they also provided an example of a non-hypercyclic distributionally chaotic operator. In 2013, Bernardes et al. \cite{Bernardes2013} conducted a systematic study of distributional chaos in Fréchet spaces. They established equivalent conditions (DDC) for distributional chaos and proved that an operator \(T\) is distributionally chaotic if and only if there exists a distributionally irregular vector.

We recall the following important notions.

\begin{itemize}
\item The \emph{upper} and \emph{lower density} of a subset \(A \subset \mathbb{N}\) are defined as
\[
\overline{\operatorname{dens}}(A) := \limsup_{N \to \infty} \frac{\operatorname{card}\{1 \leq n \leq N \,:\, n \in A\}}{N}
\]
and
\[
\underline{\operatorname{dens}}(A) := \liminf_{N \to \infty} \frac{\operatorname{card}\{1 \leq n \leq N \,:\, n \in A\}}{N},
\]
respectively.

\item An operator \(T\) on a Banach space \(X\) is said to be \emph{distributionally chaotic} if there exist an uncountable set \(\Gamma \subset X\) and \(\varepsilon > 0\) such that for every \(\delta > 0\) and each pair of distinct points \(x, y \in \Gamma\), we have that
\[
F_{x,y}(\varepsilon) := \underline{\operatorname{dens}}\big(\{j \in \mathbb{N} : \|T^j x - T^j y\| < \varepsilon\}\big) = 0
\]
and
\[
F^*_{x,y}(\delta) := \overline{\operatorname{dens}}\big(\{j \in \mathbb{N} : \|T^j x - T^j y\| < \delta\}\big) = 1.
\]
In this case, the set \(\Gamma\) is called a \emph{distributionally} \(\varepsilon\)\emph{-scrambled set} and the pair \((x, y)\) a \emph{distributionally chaotic pair}. Moreover, \(T\) is \emph{densely distributionally chaotic} if the set \(\Gamma\) may be chosen to be dense in \(X\).

\item Let \(T\) be an operator on a Banach space \(X\) and \(x \in X\). The orbit of \(x\) is said to be \textit{distributionally near to 0} if there exists \(A \subseteq \mathbb{N}\) with
\[
\overline{\operatorname{dens}}(A) = 1 \quad \text{such that} \quad \lim_{n \in A} \|T^n x\| = 0.
\]
The orbit of \(x\) is said to be \textit{distributionally unbounded} if there exists \(B \subseteq \mathbb{N}\) with
\[
\overline{\operatorname{dens}}(B) = 1 \quad \text{such that} \quad \lim_{n \in B} \|T^n x\| = \infty.
\]
If the orbit of \(x\) is simultaneously \textit{distributionally near to 0} and \textit{distributionally unbounded}, then \(x\) is called a \textit{distributionally irregular vector} for \(T\).
\end{itemize}

In this paper, we focus on distributionally chaotic composition operators on \(L^{p}\)-spaces. Composition operators on \(L^{p}\)-spaces have been the object of intense study in recent years. The topological transitivity, topological mixing, Li-Yorke chaos, \(\mathcal{F}\)-transitivity, and d\(\mathcal{F}\)-transitivity properties of this class of operators have been investigated in \cite{Bayart2018, Bernardes, He2021}. 

Throughout this paper, \((X,\mathcal{B},\mu)\) denotes a \(\sigma\)-finite measure space with a positive measure \(\mu\). The space \(L^{p}(X,\mathcal{B},\mu)\), \(1\leq p <\infty\), denotes the space of complex \(p\)-integrable functions on \((X,\mathcal{B},\mu)\), equipped with the norm
\[
\|f\| = \left(\int_{X}|f(x)|^{p} \,d\mu\right)^{\frac{1}{p}}.
\]
An inducing mapping \(\varphi:X\to X\) is a measurable map such that there exists \(c>0\) with
\[
\mu\left(\varphi^{-1}(B)\right) \leq c \mu(B), \quad \forall B\in \mathcal{B}.
\]
This condition ensures the boundedness of the composition operator \(T_{\varphi}:f \mapsto f \circ \varphi\) on \(L^{p}(X,\mathcal{B},\mu)\), \(1\leq p <\infty\).

The aim of this paper is to provide characterizations of distributionally chaotic composition operators \(T_{\varphi}\) on \(L^{p}\)-spaces in terms of \(\varphi\) and certain measurable sets. The paper is organized as follows.

 In Section~2, we provide a characterization of distributionally chaotic composition operators \(T_{\varphi}\) on \(L^{p}\)-spaces. As a consequence, we derive that \(T_{\varphi}\) is distributionally chaotic on \(L^{p}(X,\mathcal{B},\mu)\) if and only if it is distributionally chaotic on \(L^{p'}(X,\mathcal{B},\mu)\) for any \(p, p' \in [1,\infty)\). We also provide a sufficient condition, easier to verify, for demonstrating distributional chaos of \(T_{\varphi}\). Furthermore, we show that the existence of a dense set of distributionally irregular vectors implies the existence of a dense distributionally chaotic set, without any additional condition, and we provide a useful criterion for densely distributional chaos.

In Section~3, we characterize the weight sequences ensuring distributional chaos for bilateral and unilateral backward and forward shifts on weighted \(\ell^{p}\)-spaces. As a consequence, we establish the equivalence between distributional chaos and densely distributional chaos for these shifts without additional conditions.

In Section~4, we characterize composition operators \(T_{\varphi}\) on \(L^{p}(\mathbb{T},\mathcal{B},\lambda)\) induced by automorphisms of the unit disk \(\mathbb{D}\), proving that \(T_{\varphi}\) is densely distributionally chaotic if and only if \(\varphi\) has no fixed point in \(\mathbb{D}\).
\section{Characterizations of distributionally chaotic composition operators on $L^{p}$-spaces}
In this section, we provide a characterization of distributional chaotic composition operators $T_{\varphi}$ on $L^{p}$-spaces in terms of $\varphi$ and certain measurable sets in Theorem \ref{dcsufandnec}. As a result, we derive  Corollary \ref{dclp} which states that the composition operator $T_{\varphi}$ on $L^{p}(X,\mathcal{B},\mu)$ is distributionally chaotic, if and only if it is distributionally chaotic on \(L^{p'}(X,\mathcal{B},\mu)\) for any $p,p'\in [1,\infty)$. As another result, we can obtain a sufficient condition in Theorem \ref{dcsuf}, which is more easy to validate, to demonstrate that the composition operator on $L^{p}$-spaces exhibits distributional chaos.

At first, let us recall the well-known Distributional Chaos Criterion (DCC) which was introduced in \cite{Bernardes2013}. In Theorem 12 \cite{Bernardes2013}, Bernardes Jr. et al. proved that an operator \( T \) satisfies (DCC) if and only if \( T \) is distributionally chaotic, if and only if \( T \) has a distributionally irregular vector. We slightly modify the original version of Distributional Chaos Criterion (DCC) as follows, which remains equivalent to distributional chaos. Since the proof of this equivalence follows the same reasoning as that provided in Theorem 12 \cite{Bernardes2013}, we omit the proof here for brevity.
\begin{definition}[Distributional Chaos Criterion (DCC)]\label{DCC}
An operator $T$ on a Banach space $X$ satisfies the \textit{Distributional Chaos Criterion (DCC)} if there exist sequences \((x_{i})_{i\in I}\) and \( y_k \in \text{span}\{x_i : i \in I\} \), where \(I\) is a countable set, such that:
\begin{enumerate}[label=(\alph*)]
    \item There exists \( A \subseteq \mathbb{N} \) with \( \overline{\text{dens}}(A) = 1 \) such that \( \lim_{n \in A} T^n x_k = 0 \) for all \( k\in  I\).
    \item \( \lim_{k \to \infty} y_k = 0 \), and there exist \( \varepsilon > 0 \) and an increasing sequence \( (N_k) \) in \( \mathbb{N} \) such that
    \[
    \text{card} \left\{ 1 \leq j \leq N_k : d(T^j y_k, 0) > \varepsilon \right\} \geq N_k \varepsilon
    \]
 for all \( k \in \mathbb{N} \).
\end{enumerate}
\end{definition}
\begin{lemma}[Theorem 12 \cite{Bernardes2013}]
If \( T \in B(X) \), then the following assertions are equivalent:
\begin{enumerate}[label=\rm{(\roman*)}]
    \item \( T \) satisfies (DCC);
    \item \( T \) has a distributionally irregular vector;
    \item \( T \) is distributionally chaotic;
    \item \( T \) admits a distributionally chaotic pair.
\end{enumerate}
\end{lemma}
\begin{theorem}\label{dcsufandnec}
The composition operator $T_{\varphi}: L^{p}(X,\mathcal{B},\mu)\rightarrow L^{p}(X,\mathcal{B},\mu) $ is distributionally chaotic, if and only if there exist countable positive measurable subsets $\left\{B_{k,j}\right\}_{k\in \Nb,j\in S_{k}}$, where $S_{k}$ is a finite subset of $\Nb$ and $B_{k,j}\cap B_{k,j'}=\emptyset$, for any $k\in \Nb$ and $j\neq j'$,  satisfying the following conditions:
\begin{enumerate}[label=\rm{(\roman*)}]
\item there exists $A \subseteq \mathbb{N}$ with $\overline{\text{dens}}(A) = 1$ such that $\lim_{n \in A} \mu \left(\varphi^{-n}\left(B_{k,j}\right)\right) = 0$ for all $k\in \Nb,j\in S_{k}$;
\item there exist \( \varepsilon > 0 \) , an increasing sequence $(N_k)$ in $\mathbb{N}$ and countable non-zero complex numbers $(C_{k,j})_{k\in \Nb,j\in S_{k}}$ such that
    \[
    \text{card}\left\{1 \leq n \leq N_k :\frac{\sum_{j\in S_{k}}|C_{k,j}|\mu\left(\varphi^{-n}\left(B_{k,j}\right)\right)}{\sum_{j\in S_{k}}|C_{k,j}|\mu\left(B_{k,j}\right)}\geq k\right\} \geq N_k \varepsilon.
    \]
\end{enumerate}
\end{theorem}
\begin{proof}[Proof of Theorem \ref{dcsufandnec}]
We first prove the sufficiency. Denote 
\[
x_{k,j}:=\mathcal{X}_{B_{k,j}},k\in \Nb,j\in S_{k},
\]
\[z_{k}:=\sum_{j\in S_{k}}(C_{k,j})^{1/p}\mathcal{X}_{B_{k,j}},k\in \Nb,\]
and 
\[y_{k}:=\frac{z_{k}}{k^{1/p}\Vert z_{k}\Vert},k\in \Nb.\]

For $n$ in Condition (ii), we can see that 
\[\Vert T_{\varphi}^{n}y_{k}\Vert=\frac{1}{k}\left(\frac{\sum_{j\in S_{k}}|C_{j}|\mu\left(\varphi^{-n}\left(B_{k,j}\right)\right)}{\sum_{j\in S_{k}}|C_{j}|\mu\left(B_{k,j}\right)}\right)^{1/p}\geq 1.\]

Note that \((x_{k,j})_{k\in\Nb,j\in S_{k}}\) is countable and \( y_k \in \text{span}\{x_k : k \in I\} \). Then \(T_{\varphi}\) satisfies Distributional Chaos Criterion (Definition \ref{DCC}) for \((x_{k,j})\) and \((y_{k})\). 

Next we show the necessity. Suppose $T_{\varphi}$ is distributionally chaotic and let $f\in L^{p}(X,\mathcal{B},\mu)$ be a distributionally irregular vector for $T_{\varphi}$. There exist \( A,B \subseteq \mathbb{N} \) with \(\overline{\text{dens}}(A) = \overline{\text{dens}}(B) = 1\), such that
\[\lim_{n \in A} \|T_{\varphi}^{n}f\| = 0,\quad \lim_{n \in B} \|T_{\varphi}^{n}f\| = \infty.
\]

For any $k\in \Nb,$ we have 
\[
\overline{\text{dens}} \left\{ j \in \mathbb{N} :\|T_{\varphi}^{n}f\| \geq \|f\|k^{1/p} \right\} 
\geq \overline{\text{dens}}(B) = 1.
\]

Then we can choose $N_k\in \Nb$ such that
   \[
    \text{card}\left\{1 \leq n \leq N_k :\|T_{\varphi}^{n}f\| \geq \|f\|k^{1/p}\right\} \geq N_k \left(1 - k^{-1}\right).
    \] 

Since $f\in L^{p}(X,\mathcal{B},\mu)$, there exists a measurable simple function
\[s_{k}=\sum_{i\in S_{k}}C_{k,j}^{1/p}\mathcal{X}_{B_{k,j}}\text{ with }|s|\leq |f|, \] 
that approximates $f$, where $S_{k}$ is a finite set of $\Nb$, $\left\{C_{k,j}\right\}_{j\in S_{k}}\subset \mathbb{C}/\{0\}$ and $\left\{B_{k,j}\right\}_{j\in S_{k}}$ is a collection of pairwise disjoint measurable sets of finite positive measure, such that 
\begin{equation}\label{sf}
    \text{card}\left\{1 \leq n \leq N_k :\|T_{\varphi}^{n}s\| \geq \|f\|k^{1/p}\right\} \geq N_k \left(1 - k^{-1}\right).
\end{equation}

Now we deduce assertion (i) and assertion (ii) in the theorem.

Since \(s(x)\leq f(x)\) for any $x\in X,$ we have that 
\[\|T_{\varphi}^{n}s\|^{p}=\sum_{j\in S_{k}}|C_{k,j}|\mu\left(\varphi^{-n}\left(B_{k,j}\right)\right)\leq \|T_{\varphi}^{n}f\|^{p}.\]

Hence we can obtain 
\[\lim_{n \in A} \mu \left(\varphi^{-n}\left(B_{k,j}\right)\right) \leq \frac{1}{|C_{k,j}|}\lim_{n \in A} \|T_{\varphi}^{n}f\|^{p}= 0,\]
which shows assertion (i) holds true.

On the other hand, since \(\|s\|\leq \|f\|\), from inequality (\ref{sf}), we have that 
\begin{equation}\label{sp}
    \text{card}\left\{1 \leq n \leq N_k :\|T_{\varphi}^{n}s\| \geq \|s\|k^{1/p}\right\} \geq N_k \left(1 - k^{-1}\right).
\end{equation}

Note that 
\[\|s\|^{p}=\sum_{j\in S_{k}}|C_{k,j}|\mu\left(B_{k,j}\right),\|T_{\varphi}^{n}s\|^{p}=\sum_{j\in S_{k}}|C_{k,j}|\mu\left(\varphi^{-n}\left(B_{k,j}\right)\right).\]
Then we can deduce assertion (ii) in the theorem and complete the proof.
\end{proof}
By the above theorem, we can immediately derive the following conclusion.
\begin{corollary}\label{dclp}
The composition operator $T_{\varphi}: L^{p}(X,\mathcal{B},\mu)\rightarrow L^{p}(X,\mathcal{B},\mu) $ is distributionally chaotic, if and only if the composition operator $T_{\varphi}: L^{p'}(X,\mathcal{B},\mu)\rightarrow L^{p'}(X,\mathcal{B},\mu) $ is distributionally chaotic for any $p,p'\in [1,\infty).$
\end{corollary}
The condition (ii) in Theorem \ref{dcsufandnec} appears to be difficult to verify. Therefore, we present the following sufficient condition, which is more easy to  validate, to demonstrate that the composition operator on $L^{p}$-spaces exhibits distributional chaos.
\begin{proposition}\label{dcsuf}
The composition operator $T_{\varphi}: L^{p}(X,\mathcal{B},\mu)\rightarrow L^{p}(X,\mathcal{B},\mu) $ is distributionally  chaotic, if there is a non-empty collection $\left\{B_{k}\right\}_{k\in \Nb}$ of measurable sets of finite positive measure such that:
\begin{enumerate}[label=\rm{(\roman*)}]
\item there exists $A \subseteq \mathbb{N}$ with $\overline{\text{dens}}(A) = 1$ such that $\lim_{n \in A} \mu \left(\varphi^{-n}\left(B_{k}\right)\right) = 0$ for all $k\in \Nb$;
\item there exist \( \varepsilon > 0 \) and an increasing sequence $(N_k)$ in $\mathbb{N}$ such that
    \[
    \text{card}\left\{1 \leq n \leq N_k :\frac{\mu\left(\varphi^{-n}\left(B_{k}\right)\right)}{\mu\left(B_{k}\right)}\geq k\right\} \geq N_k \varepsilon .
    \]
\end{enumerate}
\end{proposition}
\begin{proof}[Proof of Proposition \ref{dcsuf}]
Note that condition $\rm(ii)$ in Propositon \ref{dcsuf} implies condition $\rm(ii)$ in Theorem \ref{dcsufandnec} by taking $S_{k}=\{k\}$, $\left\{B_{k,j}\right\}_{j\in S_{k}}=\{B_{k}\}$ and $C_{k,j}=1$.
\end{proof}
In Section 3 of \cite{Bernardes2013} and Section 6 of \cite{Bernardes}, the authors provide several sufficient conditions for densely distributionally chaotic operator $T$ on a Banach space, under which distributional chaos is equivalent to the densely distributional chaos. These sufficient conditions are based on a major premise: 
\begin{equation}\label{premise0}
T^n x \to 0 \quad \text{for all } x \in X_0,
\end{equation}
where \( X_0 \) is a dense subset of \( X \).
The method there is to construct a dense uniformly distributionally irregular manifold using the above condition (4). Naturally, this manifold is then a dense distributionally chaotic set. We recall the main theorem in Section 3 of \cite{Bernardes2013}.
\begin{theorem}[Theorem 15 \cite{Bernardes2013}]
Assume \( X \) is separable. Suppose that \( T \in B(X) \) satisfies
\[
T^n x \to 0 \quad \text{for all } x \in X_0,
\]
where \( X_0 \) is a dense subset of \( X \). Then the following assertions are equivalent:
\begin{enumerate}
    \item[(i)] \( T \) is distributionally chaotic;
    \item[(ii)] \( T \) is densely distributionally chaotic;
    \item[(iii)] \( T \) admits a dense uniformly distributionally irregular manifold;
    \item[(iv)] \( T \) admits a distributionally unbounded orbit.
\end{enumerate}
\end{theorem}
However, in the following theorem, we will show that the existence of a dense set of distributionally irregular vectors implies the existence of a dense distributionally chaotic set.
\begin{theorem}
Let  \( T \) be an operator on a separable Banach space \( X \). Then the following assertions are equivalent:
\begin{enumerate}
    \item[(i)]  \( T \) is densely distributionally chaotic ;
    \item[(ii)] the set of distributionally irregular vectors is dense in  \( X \);
    \item[(iii)]  $T$ admits a distributionally unbounded orbit and the set of vectors with orbits distributionally near to 0 is dense in  \( X \).
\end{enumerate}
\end{theorem}
\begin{proof}
The implications \((i)\Rightarrow (ii)\Rightarrow (iii)\) is trival. \((iii)\Rightarrow (ii)\) follows from Proposition 8 and Proposition 9 in \cite{Bernardes2013}. 

Now we show that \((ii)\Rightarrow (iii)\):Let \(X_{0}\) denote the set of vectors whose orbits are distributionally near to \(0\), and let \(Y_{0}\) represent the set of distributionally irregular vectors. To facilitate a detailed analysis of the properties of \(Y_{0}\), we express it as:
\[
Y_{0} = \bigcap_{k, N \in \mathbb{N}} \left( \bigcup_{m \geq N} M_{k,m} \cap \bigcup_{m \geq N} A_{k,m} \right),
\]
where
\[
M_{k,m} := \left\{ x \in X : \exists n \geq m \, \text{such that} \, 
\text{card} \left\{ 1 \leq j \leq n : \|T^j x\| < k^{-1} \right\} \geq n(1 - k^{-1}) \right\},
\]
and
\[
A_{k,m} := \left\{ x \in X : \exists n \geq m \, \text{such that} \, 
\text{card} \left\{ 1 \leq j \leq n : \|T^j x\| > k \right\} \geq n(1 - k^{-1}) \right\}.
\]
Note that each \(M_{k,m}\) and \(A_{k,m}\) is open and dense, since \(X_0\subset M_{k,m}\) and \(Y_{0}\subset A_{k,m}\). Then we have that \[U_{k,m}:=\bigcup_{m \geq N} M_{k,m} \cap \bigcup_{m \geq N} A_{k,m} \] is open and dense. So each set \(V_{k,m} := \{(x, y) \in X \times X : x - y \in U_{k,m}\}, \, j = 1, 2, \ldots\) is open and dense in \(X \times X\), and every point in \(\bigcap_{k, N \in \mathbb{N}}V_{k,m}\) is a distributionally chaotic pair for \(T\). Let \(R :=\bigcap_{k, N \in \mathbb{N}}V_{k,m}\). From Mycielski theorem in \cite{Mycielski}, there exists a dense Mycielski set (a Mycielski set is uncountable) \(K \subset X\) such that \(K \times K \setminus \Delta_X \subset R\). Then \(K\) is a dense distributionally chaotic set for \(T\).

\end{proof}
Next we give a sufficient condition on $\varphi$ for $T_{\varphi}$ to be densely distributionally chaotic, by which we can show that the composition operator $T_{\varphi}$ induced by an autormorphism $\varphi$ of the unit disk $\mathbb{D}$ on $L^{p}(\mathbb{T},d\lambda)$ is densely distributionally chaotic if and only if $\varphi$ has no fixed point in $\mathbb{D}$.
\begin{proposition}\label{dcpropdense}
Let $(X,\mathcal{B},\mu)$ be a measure space and $T_{\varphi}: L^{p}(X,\mathcal{B},\mu)\rightarrow L^{p}(X,\mathcal{B},\mu) $ be a composition operator. Suppose that there exist $A \subseteq \mathbb{N}$ with $\overline{\text{dens}}(A) = 1$ and a nonempty countable family $(C_{i})_{i\in \Nb}$ of measurable sets of finite positive measure such that $\mu\left(X\backslash\bigcup_{i\in \Nb}C_{i}\right)=0$ and that $\lim_{n \in A} \mu \left(\varphi^{-n}\left(C_{i}\right)\right) = 0$ for all $i\in \Nb$. Then the set of all vectors with orbits distributionally near to 0 is residual and the following assertions are equivalent:
\begin{enumerate}[label=\rm{(\roman*)}]
\item $T_{\varphi}$ is densely distributionally chaotic; 
\item $T_{\varphi}$ is distributionally chaotic;
\item $T_{\varphi}$ admits a distributionally unbounded orbit.
\item condition (ii) in Theorem \ref{dcsufandnec} holds.
\end{enumerate}
\end{proposition}
\begin{proof}[Proof of Proposition \ref{dcpropdense}]
In the proof of this proposition, our main task is to show that \[X_{0}:=\left\{f\in L^{p}(X,\mathcal{B},\mu): \lim_{n \in A} \|T_{\varphi}^{n}f\| = 0\right\}\] is dense in \(L^{p}(X,\mathcal{B},\mu)\). Once proven, by applying Proposition 9 in \cite{Bernardes2013}, the set of all vectors with orbits distributionally near to 0 is residual.

Now we show that $X_{0}$ is dense in $L^{p}(X,\mathcal{B},\mu).$

At first, let $G$ be the set of of all measurable simple functions on \( X \) with
\[
\mu\left(\{x : g(x) \neq 0\}\right) < \infty.\]

From Theorem 3.13 \cite{Rudin1987}, \( G \) is dense in \( L^{p}(X,\mathcal{B},\mu) \). For any $f \in L^{p}(X,\mathcal{B},\mu) $ and any finite subset $S\subset \Nb$, define a new function $f_{S}\in L^{p}(X,\mathcal{B},\mu) $ as follows
\begin{equation*}
f_{S}(x):=\left\{
\begin{aligned}
&f(x),            &x\in \{x : f(x) \neq 0\}\cap (\cup_{i\in S}C_{i}),\\
&0,         &\text{otherwise}.
\end{aligned} \right.
\end{equation*}

It is not hard to check that $G':=\left\{g_{S}:g\in G,  S \text{ is a finite subset of } \Nb\right\}$ is dense in $L^{p}(X,\mathcal{B},\mu)$.

Since any function in $G$ is bounded, then for any given $g_{S}\in G'$, denote $M :=\sup\left\{|g_{S}(x)|:x\in X\right\}<\infty$. Hence we have that
\begin{equation}\label{dse1}
\Vert T_{\varphi}^{n}g_{S}\Vert \leq M\sum_{i\in S}\left(\mu\left(\varphi^{-n}\left(C_{i}\right)\right)\right)^{1/p},n\in \Nb.
\end{equation}

From the condition given in the proposition, it is easy to see that 
\begin{equation}\label{dse2}
\lim\limits_{n\in A}\sum_{i\in S}\left(\mu\left(\varphi^{-n}\left(C_{i}\right)\right) \right)^{1/p}=\sum_{i\in S}\lim\limits_{n\in A}\left(\mu\left(\varphi^{-n}\left(C_{i}\right)\right) \right)^{1/p}=0
\end{equation}

Then we have that \[\lim\limits_{n\in A}\Vert T_{\varphi}^{n}g_{S}\Vert =0 .\] 

Hence, $G'\subset X_{0}$ and $X_{0}$ is dense in $L^{p}(X,\mathcal{B},\mu)$. From Proposition 9 in \cite{Bernardes2013}, we conclude that the set of all vectors with orbits distributionally near to 0 is residual in $L^{p}(X,\mathcal{B},\mu)$. 

Next we show that \((i)\Leftrightarrow (ii)\Leftrightarrow (iii)\Leftrightarrow (iv)\).

The implications \((i)\Rightarrow (ii)\Rightarrow (iii)\) are trival. The equivalence \((iii)\Leftrightarrow (iv)\) follows from the proof of Theorem \ref{dcsufandnec}, and we shall not repeat the proof here. It remains to prove the implication \((iii)\Rightarrow (i)\) in the theorem.

\((iii)\Rightarrow (i)\): By Proposition 8 in \cite{Bernardes2013}, the existence of distributionally unbounded orbit implies that the set of vectors with distributionally unbounded orbit is residual in \(L^{p}(X,\mathcal{B},\mu)\). Then, if the set of all vectors with orbits distributionally near to 0 is residual, (iii) implies that the set of all distributionally irregular vectors is residual and that $T_{\varphi}$ is densely distributionally chaotic.
\end{proof}
\section{Distributional chaos for backward shifts and forward shifts on $\ell^{p}(\Nb,v)$ or $\ell^{p}(\Zb,v)$}
In this section, we will provide characterizations on the weight sequence for distributionally chaotic backward shifts on weighted $\ell^{p}$-spaces $\ell^{p}(\Nb,v)$ and $\ell^{p}(\Zb,v)$, and forward shifts on  $\ell^{p}(\Zb,v)$. 

Recall that weighted $\ell^{p}$-spaces $\ell^{p}(\Zb,v)$ and $\ell^{p}(\Nb,v)$ are defined as follows:
\[
\ell^p(\Zb,v) = \left\{ (x_n)_{n \in \Zb} : \Vert x \Vert = \sum_{n\in\Zb}|x_n|^p v_n < \infty \right\}, \quad 1 \leq p < \infty,
\]
and 
\[
\ell^p(\Nb,v) = \left\{ (x_n)_{n \in \Nb} : \Vert x \Vert = \sum_{n\geq 1}|x_n|^p v_n < \infty \right\}, \quad 1 \leq p < \infty,
\]
where $v = (v_n)$ is a positive weight sequence. The backward shift $B$ on the weighted sequence spaces $\ell^{p}(\mathbb{Z},v)$ and $\ell^{p}(\mathbb{N},v)$  is defined as follows:
\[
B((x_n)_{n \in \mathbb{Z}}) = (x_{n+1})_{n \in \mathbb{Z}}\ (B((x_n)_{n \in \mathbb{N}}) = (x_{n+1})_{n \in \mathbb{N}}).
\] 

We point out here that $\ell^{p}(\mathbb{Z}, v)$ is essentially the $L^{p}$-space, $L^{p}(\mathbb{Z}, \mathcal{P}(\mathbb{Z}), \mu)$, studied in this paper, where $X=\mathbb{Z}$, $\mathcal{B}=\mathcal{P}(\mathbb{Z})$ (the power set of $\mathbb{Z}$), $\mu(\{j\})=v_{j},j\in \Zb$. And the backward shift $B$ on the sequence space $\ell^{p}(\mathbb{Z},v)$ is the composition operator $T_{\sigma}$ on $L^{p}(\mathbb{Z},\mathcal{P}(\mathbb{Z}),\mu)$, with the inducing mapping $\sigma:i\in \mathbb{Z}\mapsto i+1\in \mathbb{Z}$. Similarly, the space $\ell^p(\mathbb{N}, v)$ is equivalent to $L^p(\mathbb{N}, \mathcal{P}(\mathbb{N}), \mu)$, and the backward shift operator on $\ell^p(\mathbb{N}, v)$ corresponds to the composition operator $T_{\sigma}$ on $L^p(\mathbb{N}, \mathcal{P}(\mathbb{N}), \mu)$ with the inducing mapping $\sigma:i\in \mathbb{N}\mapsto i+1\in \mathbb{N}$.

The boundedness of the operator $B$ or the composition operator $T_{\sigma}$ on $\ell^{p}(\Zb,v)$ ($\ell^{p}(\Nb,v)$) is equivalent to the condition that \[
\sup_{n \in \mathbb{Z}} \frac{v_n}{v_{n+1}} < \infty \ (\sup_{n \in \mathbb{N}} \frac{v_n}{v_{n+1}} < \infty, \text{ respectively }).
\]

The weighted backward shift operator on the sequence space $\ell^{p}(\mathbb{Z})$ ($\ell^{p}(\mathbb{N})$) is defined as follows:
\[
B_{\omega}((x_n)_{n \in \mathbb{Z}}) = (w_{n+1}x_{n+1})_{n \in \mathbb{Z}}\ (B_{\omega}((x_n)_{n \in \mathbb{N}}) = (w_{n+1}x_{n+1})_{n \in \mathbb{N}}),
\]
where \( w = (w_n)_{n \in \mathbb{Z}} \) (\( w = (w_n)_{n \in \mathbb{N}} \), respectively) is a weight sequence, that is, a sequence of bounded nonzero scalars.  The boundedness of $\omega$ is to ensure that the operator $B_{\omega}$ is bounded.  
Notably, the weighted backward shift $B_{\omega}$  on $\ell^{p}(\mathbb{Z})$ ($\ell^{p}(\mathbb{N})$)  can be identified with the backward shift \(B\) on $\ell^{p}(\mathbb{Z},v)$ ($\ell^{p}(\mathbb{N},v)$)  where 
\[
v_n =
\begin{cases}
 \prod_{i=1}^n |w_i |^{-p}, & \text{for } n \geq 1, \\
1, & \text{for } n = 0, \\
\prod_{i=n+1}^0 |w_i|^{p}, & \text{for } n \leq -1,
\end{cases}
\]

\[
\left(v_n=\prod_{i=1}^n |w_i| ^{-p},  \text{for } n \geq 1, \text{respectively.}\right)
\]

Since the weighted backward shift operator on \(\ell^p\) space is topologically equivalent to the backward shift operator on weighted \(\ell^p\) space, all conclusions about the backward shift operator on weighted \(\ell^p\) space in this section also hold correspondingly for the weighted backward shift operator on \(\ell^p\) space. 
\subsection{Distributionally chaotic backward shifts on $\ell^{p}(\Nb,v)$ }
\begin{theorem}\label{bisufnec}
Suppose that the backward shift $B$ is an operator on $\ell^{p}(\mathbb{N}, v)$.  Then $B$ is distributionally chaotic if and only if it is densely distributionally chaotic, if and only if there exist  \( \varepsilon > 0 \), an increasing sequence $(N_k)$ in $\mathbb{N}$, countable finite subsets $\left\{S_{k}\right\}_{k\in \Nb}$ of $\Nb$ and countable complex numbers $(C_{k,j})_{k\in \Nb,j\in S_{k}}$ such that
    \[
    \text{card}\left\{1 \leq n \leq N_k :\frac{\sum_{j\in S_{k}}|C_{k,j}|v_{-n+j}}{\sum_{j\in S_{k}}|C_{k,j}|v_{j}}\geq k\right\} \geq N_k \varepsilon.
    \]
\end{theorem}
Next we give a sufficient condition for the backward shift $B$ is an operator on $\ell^{p}(\mathbb{N}, v)$ to be distributionally chaotic, i.e., densely distributionally chaotic.
\begin{theorem}\label{ubsuf}
Suppose that the backward shift $B$ is an operator on $\ell^{p}(\mathbb{N}, v)$.  Then $B$ is distributionally chaotic if and only if it is densely distributionally chaotic, if there exist \( \varepsilon > 0 \), an increasing sequence $(N_k)$ in $\mathbb{N}$, countable finite subsets $\left\{S_{k}\right\}_{k\in \Nb}$ of $\Nb$ such that
\begin{equation}\label{ubsufequa}
    \text{card}\left\{1 \leq n \leq N_k :\frac{\sum_{j\in S_{k}}v_{-n+j}}{\sum_{j\in S_{k}}v_{j}}\geq k\right\} \geq N_k \varepsilon.
\end{equation}
\end{theorem}
\begin{remark}
We remark that for the backward shift $B$ on  $\ell^{p}(\mathbb{N},v)$, the condition in this theorem is  strictly weaker (see Example \ref{exabac} below) than that in  Theorem 20 \cite{Bernardes2013}, which states that,  $B$ is distributionally chaotic if there exists a set \( S \subseteq \mathbb{N} \) with \( \overline{\text{dens}}(S) =1 \) such that
\[
\sum_{n \in S} v_{n} <\infty.
\]
\end{remark}
In Theorem 27 \cite{Bernardes},  the unilateral weighted backward shift $B_{\omega}$  on $\ell^{p}(\mathbb{N})$ with $1\leq p< \infty$, where $\omega_{k}:=(\frac{k}{k-1})^{1/p},k\geq 2$, was shown distributionally chaotic. Notably, this weighted backward shift $T$  on $\ell^{p}(\mathbb{N})$ can be identified with the backward shift $B$ on $\ell^{p}(\mathbb{N}, v),$ where  $v_{k}=1/k,k\in \Nb$.  Using Theorem \ref{ubsuf}, we can further illustrate, in detail, how the operator exhibits distributional chaos. 
\begin{example}\label{exabac}
Let $v_{k}=1/k,k\in \Nb$. Then the backward shift $B$ on $\ell^{p}(\mathbb{N}, v)$ is distributionally chaotic.
\end{example}
\begin{proof}
Define \( N_k = 2k (2k)^k \) and \( S_k = [(2k)^k+1, 2k (2k)^k] \cap \mathbb{N} \), for any \(k\in \Nb\). It remains to verify that inequality (\ref{ubsufequa}) in Theorem \ref{ubsuf} holds for \( (N_k) \) and \( (S_k) \). 

Specifically, we need to show that
\begin{equation}\label{n}
\frac{\sum_{j \in S_k} v_{-n+j}}{\sum_{j \in S_k} v_j} \geq k
\end{equation}
holds for all \( (2k)^k +1\leq n \leq (2k-1)(2k)^k \). 

Observe that
\[
\text{card}\left\{ (2k)^k+1 \leq n \leq (2k-1)(2k)^k \right\} = N_k \left( 1 - k^{-1} \right).
\]
Thus, once inequality (\ref{n}) is established, it directly follows that inequality (\ref{ubsufequa}) is also proved.

For \((2k)^k+1 \leq n \leq (2k-1)(2k)^k\), we can see that \(v_{-n+j}=0,\) if \(j\leq n\) and \(v_{-n+j}=1/(j-n),\) if \(j>n\). Hence we have that 
\[
\sum_{j\in S_{k}}v_{-n+j}=\sum_{j=n+1}^{2k(2k)^k}1/(j-n)=\sum_{j=1}^{2k(2k)^k-n}1/j\geq \sum_{j=1}^{(2k)^k}1/j\geq k\ln 2k.
\]

Note that, \[
\sum_{j\in S_{k}}v_{j}=\sum_{j=(2k)^k+1}^{2k(2k)^k}1/j \leq  \ln(2k(2k)^k)-\ln ((2k)^k)=\ln2k.
\]

Now we can deduce that inequality (\ref{n}) holds and complete the proof.
\end{proof}
\begin{remark}\label{remarklv}
We remark that the above example satisfies the condition in Theorem \ref{ubsuf}, but there exists no set \( S \subseteq \mathbb{N} \) with \( \overline{\text{dens}}(S) =1 \) such that
\[
\sum_{n \in S} v_{n} <\infty.
\]

\end{remark}

\begin{corollary}\label{coruws1}
Let $\omega=(\omega_{n})_{n\in \mathbb{N}}$ be a bounded sequence of nonzero scalars. The (unilateral) weighted backward shift $B_{\omega}$ on  $\ell^{p}(\mathbb{N})$ is distributionally chaotic if and only it is densely distributionally chaotic, if there exist \( \varepsilon > 0 \), an increasing sequence $(N_k)$ in $\mathbb{N}$, countable finite subsets $\left\{S_{k}\right\}_{k\in \Nb}$ of $\Nb$ such that
\begin{equation}
    \text{card}\left\{1 \leq n \leq N_k :\frac{\sum_{j\in S_{k}}|w_{1}w_{2}\cdots w_{j-n}|^{-p}}{\sum_{j\in S_{k}}|w_{1}w_{2}\cdots w_{j}|^{-p}}\geq k\right\} \geq N_k  \varepsilon.
\end{equation}
\end{corollary}
\begin{remark}
We remark that for the weighted backward shift $B_{\omega}$ on  $\ell^{p}(\mathbb{N})$, the condition in this theorem is  strictly weaker than that in  Corollary 40 \cite{Bernardes2018}, which states that,  a weighted backward shift $B_{\omega}$ on  $\ell^{p}(\mathbb{N})$ is distributionally chaotic if there exists a set \( S \subseteq \mathbb{N} \) with \( \overline{\text{dens}}(S) > 0 \) such that
\begin{equation}\label{coro40}
\sum_{n \in S} \prod_{\nu=1}^{n}| \omega_\nu |^{-1}  <\infty.
\end{equation}
The example provided in Theorem 27 \cite{Bernardes} satisfies Corollary \ref{coruws1}, but there exists no set \( S \subseteq \mathbb{N} \) with \( \overline{\text{dens}}(S) >0 \) such that condition (\ref{coro40}) holds.
\end{remark}
\subsection{Distributionally chaotic backward shifts on weighted $\ell^{p}$-spaces $\ell^{p}(\mathbb{Z}, v)$. }
In this subsection,  we address the characterizations of distributionally chaotic (bilateral) backward shifts on weighted $\ell^{p}$-spaces $\ell^{p}(\mathbb{Z}, v)$, which is little known in the current literature.  More specifically, we prove that a (bilateral) backward shift $B$ on a  weighted $\ell^{p}$-space $\ell^{p}(\mathbb{Z}, v)$ is distributionally chaotic, if and only if it is densely distributionally chaotic, if and only if the weight sequence satisfies the condition in Theorem \ref{bisufnec}.

At first, let us give an important lemma which plays a key role in the proof of Theorem \ref{bi0} and Theorem \ref{bisufnec}.
\begin{lemma}\label{1}
Let \((X, \mathcal{B}, \mu)\) be a measure space, \(I\) a countable set, and \(\{B_{i}\}_{i \in I}\) a sequence of positive, measurable subsets. Then the following conditions are equivalent.
\begin{enumerate}[label=\rm{(\roman*)}]
    \item There exists a set \(A \subseteq \mathbb{N}\) with \(\overline{\mathrm{dens}}(A) = 1\) such that 
    \[
    \lim_{n \in A} \mu \left( \varphi^{-n} \left( B_{i} \right) \right) = 0 \quad \text{for all } i \in I.
    \]
    \item There exists an increasing sequence \((M_k)\) in \(\mathbb{N}\) such that, for any \(k \in \mathbb{N}\) and \(i \in I\), there exists a corresponding \(n_{k,i}\) satisfying 
    \begin{equation}\label{mj}
        \text{card} \left\{ 1 \leq n \leq M_{j} : \mu \left( \varphi^{-n} \left( B_{i} \right) \right) < k^{-1} \right\} \geq M_{j} \left( 1 - k^{-1} \right)
    \end{equation}
    for all \(j \geq n_{k,i}\).
\end{enumerate}
\end{lemma}

\begin{proof}[Proof of Lemma \ref{1}]
$(i)\Rightarrow (ii)$ is trival. We only need to show that $(ii)\Rightarrow (i).$ Since $I$ is a countable set, it can be written as $\left\{i_1,i_2,i_3 \dots\right\}.$ Without loss of generality, we assume that $I=\Nb.$ We denote by
\begin{equation}\label{Ajlk}
  A_{j,k,i}= \left\{1 \leq n \leq M_{j} :\mu \left(\varphi^{-n}\left(B_{i}\right)\right)<(k)^{-1}\right\}, j,k,i\in \Nb.
\end{equation}
Then from \ref{mj}, we have that 
\begin{equation}\label{cardAjlk}
\text{card}(A_{j,k,i})\geq M_{j}(1-(k)^{-1})
\end{equation}
 for any $j\geq n_{k,i}.$

We can define inductively sequences $(m_{k})$ in $\Nb$ and $(A_{k})$ in $\mathcal{P}(\mathbb{N})$ satisfying the following conditions:
\begin{enumerate}[label=\rm{(\arabic*)}]
\item $m_{1}=n_{1,1}$, $A_{1}=A_{1,1,1}$;
\item $ m_{k}\geq \max\{n_{2k^{2},i},1\leq i\leq k \}$ and $M_{m_{k-1}}/M_{m_{k}}<(2k)^{-1}$, for all $k\geq 2$;
\item $A_{k}\subset [M_{m_{k-1}}+1,M_{m_{k}}]$ and $\text{card}(A_{k})\geq M_{m_{k}}\cdot (1-k^{-1})$, for all $k\geq 2$;
\item $\mu \left(\varphi^{-n}\left(B_{i}\right)\right)<(k)^{-1},1\leq i\leq k,n\in A_{k},k\in \Nb$
\end{enumerate} 
Once the construction is complete, define \( A = \bigcup_{k} A_{k} \). It follows that 
\[
\overline{\mathrm{dens}}(A) = 1 \quad \text{and} \quad \lim_{n \in A} \mu \left( \varphi^{-n}(B_{i}) \right) = 0, \quad \text{for all } i \in \mathbb{N}.
\]
This establishes the implication \((\mathrm{ii}) \Rightarrow (\mathrm{i})\).

The construction of the sequences $(m_k)$ and $(A_k)$ proceeds inductively as follows:  
\begin{enumerate}
    \item For $k = 1$, set $m_{1} = n_{1,1}$ and $A_{1} = A_{1,1,1}$. From \ref{Ajlk} and \ref{cardAjlk}, it is clear that
\begin{equation*}
\text{card}(A_{1})\geq M_{m_{1}}(1-1^{-1})\text{ and }\mu \left(\varphi^{-n}\left(B_{1}\right)\right)<1^{-1},n\in A_{1}
\end{equation*}
    \item For $k = 2$, choose $m_{2} \geq \max\{n_{8,1}, n_{8,2}\}$ such that 
    \[
    \frac{M_{m_{1}}}{M_{m_{2}}} < 4^{-1}.
    \]
    Define the set 
    \[
    A_{2} = A_{m_{2},8,1} \cap A_{m_{2},8,2} \cap [M_{m_{1}} + 1, M_{m_{2}}].
    \]
From \ref{mj} and \ref{Ajlk}, for any $n\in A_{2}$, we have that
\[
\mu \left(\varphi^{-n}\left(B_{i}\right)\right)<8^{-1}<2^{-1},i=1,2
\]
and that 
\begin{equation*}
\begin{aligned}
    \text{card}(A_{2})=\text{card}(A_{2}\cap[M_{m_{1}}+1,M_{m_{2}}])&\geq\text{card}\left\{A_{m_{2},8,1}\cap A_{m_{2},8,2}\right\}-\text{card}([1,M_{m_{1}}])\\
&\geq M_{m_{2}}(1-(8)^{-1}-(8)^{-1}-4^{-1})\\
&=M_{m_{2}}\cdot (1-2^{-1}).
\end{aligned}
\end{equation*}
    \item Assume that $m_{k-1}$ and $A_{k-1}$ have been constructed and satisfy conditions (2), (3) and (4). For $k \geq 3$, select $m_{k}$ such that 
    \[
    m_{k} \geq \max\{n_{2k^{2},1}, n_{2k^{2},2}, \dots, n_{2k^{2},k}\} \quad \text{and} \quad \frac{M_{m_{k-1}}}{M_{m_{k}}} < (2k)^{-1}.
    \]
    Define the set 
    \[
    A_{k} = A_{m_{k},2k^{2},1} \cap A_{m_{k},2k^{2},2} \cap \dots \cap A_{m_{k},2k^{2},k} \cap [M_{m_{k-1}} + 1, M_{m_{k}}].
    \]
\end{enumerate}

We now verify that $A_{k}$ satisfies the required conditions (3) and (4). Specifically, we have:
\begin{equation}
\begin{aligned}
    \text{card}(A_{k}) &= \text{card}\left( A_{k} \cap [M_{m_{k-1}} + 1, M_{m_{k}}] \right) \\
    &\geq \text{card} \left( A_{m_{k},2k^{2},1} \cap A_{m_{k},2k^{2},2} \cap \dots \cap A_{m_{k},2k^{2},k} \right) 
    - \text{card}([1, M_{m_{k-1}}]) \\
    &\geq M_{m_{k}} \left( 1 - \frac{1}{2k^{2}} - \frac{1}{2k^{2}} - \dots - \frac{1}{2k^{2}} - \frac{1}{2k} \right) \\
    &= M_{m_{k}} \left( 1 - k^{-1} \right).
\end{aligned}
\end{equation}
and that
\[
\mu \left(\varphi^{-n}\left(B_{i}\right)\right)<(2k^{2})^{-1}<k^{-1},n\in A_{k},1\leq i\leq k.
\]

\end{proof}

\begin{theorem}\label{bi0}
Suppose that the backward shift $B$ is an operator on $X=\ell^{p}(\mathbb{Z}, v)$.  Then the following assertions are equivalent:
\begin{enumerate}[label=\rm{(\roman*)}]
\item the set of all vectors with orbits distributionally near to 0 is residual;
\item there exists a vector $x\in X$ with orbit distributionally near to 0; 
\item there exists $A \subseteq \mathbb{N}$ with $\overline{\text{dens}}(A) = 1$ such that $\lim_{n \in A} v_{-n} = 0$;
\end{enumerate}
\end{theorem}
\begin{proof}
The implications $(i)\Rightarrow (ii)\Rightarrow (iii)$ are trivial. We only need to show that $(iii)\Rightarrow (i).$

First, we interpret \(\ell^{p}(\mathbb{Z}, v)\) as its equivalent space \(L^{p}(\mathbb{Z}, \mathcal{P}(\mathbb{Z}), \mu)\), where \(X = \mathbb{Z}\), \(\mathcal{B} = \mathcal{P}(\mathbb{Z})\), and \(\mu(\{j\}) = v_{j}\) for all \(j \in \mathbb{Z}\). Letting $C_{i}=\{i\},i\in \Zb$, it remains to verify that $(C_{i})$ satisfies the assumptions of Proposition \ref{dcpropdense}.
The condition stated in  Proposition \ref{dcpropdense}, \(\mu\left(X \backslash \bigcup_{i \in I} C_{i}\right) = 0\), is self-evident. Our task now is to identify a set \( A \subseteq \mathbb{N} \) with \(\overline{\mathrm{dens}}(A) = 1\) such that  
\[
\lim_{n \in A} \mu \left( \varphi^{-n}(\{i\}) \right) = 0 \quad \text{for all } i \in \mathbb{Z}.
\]

Notice that, $\mu \left( \varphi^{-n} \left( \{0\} \right) \right)=v_{-n}$ and $\mu \left( \varphi^{-n} \left( \{i\} \right) \right)=v_{-n+i}$, for all $n\in \Nb,i\in \Zb$.
From Theorem \ref{1}, we only need to find an increasing sequence \((M_k)\) in \(\mathbb{N}\) such that, for any \(k \in \mathbb{N}\) and \(i \in \Zb\), there exists a corresponding \(n_{k,i}\) satisfying 
\begin{equation}\label{wbmj}
        \text{card} \left\{ 1 \leq n \leq M_{j} : v_{-n+i} < k^{-1} \right\} \geq M_{j} \left( 1 - k^{-1} \right)
\end{equation}
for all \(j \geq n_{k,i}\).
Since there exists $A \subseteq \mathbb{N}$ with $\overline{\text{dens}}(A) = 1$ such that $\lim_{n \in A} \mu \left( \varphi^{-n} \left( \{0\} \right) \right)=\lim_{n \in A} v_{-n} = 0$. Then there exists an increasing sequence \((N_k)\) in \(\mathbb{N}\) such that,
for any \(k \in \mathbb{N}\) , there exists a corresponding \(n_{k}\) satisfying 
\begin{equation*}
        \text{card} \left\{ 1 \leq n \leq N_{j} : v_{-n} < k^{-1} \right\} \geq N_{j} \left( 1 - k^{-1} \right)
\end{equation*}
for all \(j \geq n_{k}\).

Now we claim that \((N_k)\) is the desired \((M_k)\) satisfying inequality (\ref{wbmj}). 
For any $i\in \Zb$ and $k\in \Nb$, take $n_{k,i}\geq n_{k(k+1)}$ so that $n_{k,i}\cdot (k+1)^{-1}\geq |i|$.
Then we have that
\begin{equation*}
\begin{aligned}
\text{card} \left\{ 1 \leq n \leq N_{j} :v_{-n+i}< (k(k+1))^{-1} \right\} 
&\geq \text{card} \left\{ 1 \leq n \leq N_{j} : v_{-n} <(k(k+1))^{-1} \right\}-|i|\\
&\geq N_{j} \left( 1 - (k(k+1))^{-1} \right)-|i|\\
&\geq N_{j} \left( 1 - k^{-1} \right),
\end{aligned}
\end{equation*}
for any $j\geq n_{k,i}.$ Therefore, inequality (\ref{wbmj}) follows immediately.

\end{proof}
From Theorem \ref{bi0},  we can derive the equivalence between distributional chaos and densely distributional chaos for backward shifts on $\ell^{p}(\mathbb{Z},v)$ spaces and weighted backward shift $B_{\omega}$ on $\ell^{p}(\mathbb{Z})$ without any additional conditions.
\begin{corollary}
Suppose that $T$ is the backward shift $B$ on $\ell^{p}(\mathbb{Z}, v)$ or weighted backward shift $B_{\omega}$ on $\ell^{p}(\mathbb{Z})$ and that $T$ is bounded. Then $T$ is distributionally chaotic if and only it is densely distributionally chaotic.
\end{corollary}
\begin{proof}
Since a weighted backward shift $B_{\omega}$ on $\ell^{p}(\mathbb{Z})$ can be identified with a backward shift $B$ is on $\ell^{p}(\mathbb{Z}, v)$, it remains to show the equivalence between distributional chaos and densely distributional chaos for backward shifts on $\ell^{p}(\mathbb{Z},v)$ spaces.

Let $T$ be a distributional chaotic backward shift $B$ on $\ell^{p}(\mathbb{Z}, v)$. By Theorem 12 \cite{Bernardes2013}, \(T\) has a distributionally irregular vector. From Theorem \ref{bi0}, the existence of vectors with orbits distributionally near to 0, implies that the set of all vectors with orbits distributionally near to 0 is residual in \(L^{p}(X,\mathcal{B},\mu)\). By Proposition 8 in \cite{Bernardes2013}, the existence of distributionally unbounded orbit implies that the set of vectors with distributionally unbounded orbit is residual in \(L^{p}(X,\mathcal{B},\mu)\). Therefore, we conclude that 

In the proof of this proposition, our main task is to show that \[X_{0}:=\left\{f\in L^{p}(X,\mathcal{B},\mu): \lim_{n \in A} \|T_{\varphi}^{n}f\| = 0\right\}\] is dense in \(L^{p}(X,\mathcal{B},\mu)\). Once proven, by applying Proposition 9 in \cite{Bernardes2013}, the set of all vectors with orbits distributionally near to 0 is residual. 
By Proposition 8 in \cite{Bernardes2013}, the existence of distributionally unbounded orbit implies that the set of vectors with distributionally unbounded orbit is residual in \(L^{p}(X,\mathcal{B},\mu)\). Then, if the set of all vectors with orbits distributionally near to 0 is residual, (iii) implies that the set of all distributionally irregular vectors is residual and that $T_{\varphi}$ is densely distributionally chaotic.

\end{proof}
From Theorem \ref{bi0}, we can obtain the following characterization.
\begin{theorem}\label{bisufnec}
Suppose that the backward shift $B$ is an operator on $\ell^{p}(\mathbb{Z}, v)$.  Then $B$ is distributionally chaotic if and only it is densely distributionally chaotic, if and only if there exist countable finite subsets $\left\{S_{i}\right\}_{i\in \Nb}$ of $\Zb$ satisfying the following conditions:
\begin{enumerate}[label=\rm{(\roman*)}]
\item there exists $A \subseteq \mathbb{N}$ with $\overline{\text{dens}}(A) = 1$ such that $\lim_{n \in A} v_{-n} = 0$;
\item there exist \( \varepsilon > 0 \) and an increasing sequence $(N_k)$ in $\mathbb{N}$ such that
    \[
    \text{card}\left\{1 \leq n \leq N_k :\frac{\sum_{j\in S_{k}}|C_{j}|v_{-n+j}}{\sum_{j\in S_{k}}|C_{j}|v_{j}}\geq k\right\} \geq N_k  \varepsilon,
    \]
for some $C_{j}\in\mathbb{C}.$
\end{enumerate}
\end{theorem}
Next we give a sufficient condition for the backward shift $B$ on $\ell^{p}(\mathbb{Z}, v)$ to be distributionally chaotic, equivalently, densely distributionally chaotic.
\begin{theorem}\label{bisuf}
Suppose that the backward shift $B$ is an operator on $\ell^{p}(\mathbb{Z}, v)$.  Then $B$ is distributionally chaotic if and only it is densely distributionally chaotic, if there exist countable finite subsets $\left\{S_{i}\right\}_{i\in \Nb}$ of $\Zb$ satisfying the following conditions:
\begin{enumerate}[label=\rm{(\roman*)}]
\item there exists $A \subseteq \mathbb{N}$ with $\overline{\text{dens}}(A) = 1$ such that $\lim_{n \in A} v_{-n} = 0$;
\item there exist \( \varepsilon > 0 \) and an increasing sequence $(N_k)$ in $\mathbb{N}$ such that
    \[
    \text{card}\left\{1 \leq n \leq N_k :\frac{\sum_{j\in S_{k}}v_{-n+j}}{\sum_{j\in S_{k}}v_{j}}\geq k\right\} \geq N_k  \varepsilon.
    \]

\end{enumerate}
\end{theorem}

\begin{example}\label{exabibac}
Define \(n_{k}=k^{k},N_{k}=kk^{k},k\geq 1, N_{0}=0\) and 
\[
v_n =
\begin{cases}
1/n, & \text{if } n \geq 1, \\
1, & \text{if } n = 0, \\
1, & \text{if } N_{k-1} < -n\leq n_{k} \text{ for some }k\geq 1,\\
1/k, & \text{if }n_{k}+1 < -n\leq (k-1)n_{k} \text{ for some } k\geq 2.
\end{cases}
\] Then the backward shift $B$ on $\ell^{p}(\mathbb{N}, v)$ is densely distributionally chaotic.
\end{example}
\begin{proof}
Since \(v_n =1/n\), from the proof in Example \ref{exabac}, we can verify Condition (ii) in Theorem \ref{bisuf} holds. Define $A=\bigcup_{k\geq 1}[n_{k}+1,(k-1)n_{k}]$. It is not hard to see that $\overline{\text{dens}}(A) = 1$ and $\lim_{n \in A} v_{-n} = 0$.
\end{proof}
\subsection{Distributionally chaotic forward shift on $\ell^{p}(\mathbb{Z}, v)$}
In this subsection, we will provide characterizations on the weight sequence for the distributionally chaotic  forward shift on $\ell^{p}(\Zb,v)$ in Theorem \ref{uniforsufnec},  from which we reveal the equivalence between distributional chaos and densely distributional chaos for forward shifts on $\ell^{p}(\mathbb{Z},v)$ spaces without any additional conditions. As results, we also derive sufficient conditions for distributionally chaotic forward shifts on $\ell^{p}(\mathbb{Z},v)$ and weighted forward shifts on $\ell^{p}(\Zb)$ spaces.

Recall that the forward shift $F$ on the sequence space $\ell^{p}(\mathbb{Z},v)$  is defined as follows:
\[
F((x_{n})_{n \in \mathbb{Z}}) = (x_{n-1})_{n \in \mathbb{Z}}.
\] 

As in the initial discussion of the backward shift operator at the beginning of this section, the forward shift operator $F$ on the sequence space $\ell^{p}(\mathbb{Z},v)$, is indeed the composition operator $T_{\sigma}$ on $L^{p}(\mathbb{Z},\mathcal{P}(\mathbb{Z}),\mu)$, with the inducing mapping $\sigma:i\in \mathbb{Z}\mapsto i-1\in \mathbb{Z}$, where $\mu(\{j\})=v_{j},j\in \Zb$. 
The boundedness of the operator $F$ or the composition operator $T_{\sigma}$ on $\ell^{p}(\Zb,v)$ ($\ell^{p}(\Nb,v)$) is equivalent to the condition that \[
\sup_{n \in \mathbb{Z}} \frac{v_{n+1}}{v_{n}} < \infty .
\]

The weighted forward shift operator \(F_{\omega}\) on the sequence space $\ell^{p}(\mathbb{Z})$ is defined as follows:
\[
F_{\omega}((x_n)_{n \in \mathbb{Z}}) = (w_{n}x_{n-1})_{n \in \mathbb{Z}},
\]
where \( w = (w_n)_{n \in \mathbb{Z}} \)  is a weight sequence, that is, a sequence of bounded nonzero scalars.  
The weighted forward shift $F_{\omega}$  on $\ell^{p}(\mathbb{Z})$  can be identified with the forward shift \(F\) on $\ell^{p}(\mathbb{Z},v)$ ($\ell^{p}(\mathbb{N},v)$),  where 
\[
v_n =
\begin{cases}
 \prod_{i=1}^n |w_i |^{p}, & \text{for } n \geq 1, \\
1, & \text{for } n = 0, \\
\prod_{i=n+1}^0 |w_i|^{-p}, & \text{for } n \leq -1.
\end{cases}
\]

In Theorem 29 \cite{Bernardes2013}, the authors provide a sufficient condition under which the weighted forward shift $F_{\omega}$ on a Fréchet sequence space over \( \mathbb{Z} \) in which \( (e_n)_{n \in \mathbb{Z}} \) is a basis is distributionally chaotic if and only if it is densely distributionally chaotic. For the $\ell^{p}$-space $\ell^{p}(\mathbb{Z},v)$, this condition can be simplied as follows:
\begin{equation}\label{eqtheorem29}
\lim_{n \to \infty} \left( \prod_{\nu=1}^{n} \omega_\nu \right)  = 0. 
\end{equation}
However, condition (\ref{eqtheorem29}) is very strict for general weighted forward shift on $\ell^{p}(\mathbb{Z},v)$ space. In fact, we can show the equivalence between distributional chaos and densely distributional chaos for forward shifts on $\ell^{p}(\mathbb{Z},v)$ spaces and weighted forward shifts on $\ell^{p}(\mathbb{Z})$ in Theorem \ref{uniforsufnec} without any additional conditions.
\begin{theorem}\label{uniforsufnec}
Suppose that the forward shift $F$ is an operator on $\ell^{p}(\mathbb{N}, v)$.  Then $F$ is distributionally chaotic if and only it is densely distributionally chaotic, if and only if there exist countable finite subsets $\left\{S_{i}\right\}_{i\in \Nb}$ of $\Nb$ satisfying the following conditions:
\begin{enumerate}[label=\rm{(\roman*)}]
\item there exists $A \subseteq \mathbb{N}$ with $\overline{\text{dens}}(A) = 1$ such that $\lim_{n \in A} v_{n} = 0$;
\item there exist \( \varepsilon > 0 \) and an increasing sequence $(N_k)$ in $\mathbb{N}$ such that
    \[
    \text{card}\left\{1 \leq n \leq N_k :\frac{\sum_{j\in S_{k}}|C_{j}|v_{n+j}}{\sum_{j\in S_{k}}|C_{j}|v_{j}}\geq k\right\} \geq N_k \varepsilon,
    \]
for some $C_{j}\in\mathbb{C}.$
\end{enumerate}
\end{theorem}
\begin{theorem}\label{biforsufnec}
Suppose that the forward shift $F$ is an operator on $\ell^{p}(\mathbb{Z}, v)$.  Then $F$ is distributionally chaotic if and only it is densely distributionally chaotic, if and only if there exist countable finite subsets $\left\{S_{i}\right\}_{i\in \Nb}$ of $\Zb$ satisfying the following conditions:
\begin{enumerate}[label=\rm{(\roman*)}]
\item there exists $A \subseteq \mathbb{N}$ with $\overline{\text{dens}}(A) = 1$ such that $\lim_{n \in A} v_{n} = 0$;
\item there exist \( \varepsilon > 0 \) and an increasing sequence $(N_k)$ in $\mathbb{N}$ such that
    \[
    \text{card}\left\{1 \leq n \leq N_k :\frac{\sum_{j\in S_{k}}|C_{j}|v_{n+j}}{\sum_{j\in S_{k}}|C_{j}|v_{j}}\geq k\right\} \geq N_k \varepsilon,
    \]
for some $C_{j}\in\mathbb{C}.$
\end{enumerate}
\end{theorem}
Next we give a sufficient condition for the forward shift $F$ on $\ell^{p}(\mathbb{Z}, v)$ to be distributionally chaotic, equivalently, densely distributionally chaotic.
\begin{theorem}\label{biforsuf}
Suppose that the forward shift $F$ is an operator on $\ell^{p}(\mathbb{Z}, v)$.  Then $F$ is distributionally chaotic if and only it is densely distributionally chaotic, if there exist countable finite subsets $\left\{S_{i}\right\}_{i\in \Nb}$ of $\Zb$ satisfying the following conditions:
\begin{enumerate}[label=\rm{(\roman*)}]
\item there exists $A \subseteq \mathbb{N}$ with $\overline{\text{dens}}(A) = 1$ such that $\lim_{n \in A} v_{n} = 0$;
\item there exist \( \varepsilon > 0 \) and an increasing sequence $(N_k)$ in $\mathbb{N}$ such that
    \[
    \text{card}\left\{1 \leq n \leq N_k :\frac{\sum_{j\in S_{k}}v_{n+j}}{\sum_{j\in S_{k}}v_{j}}\geq k\right\} \geq N_k \varepsilon.
    \]
\end{enumerate}
\end{theorem}
\begin{remark}
By symmetrically reversing the components of the weight sequence \(v_n\) from Example \ref{exabibac} for \(n > 0\) and \(n < 0\) about \(n=0\) , we construct a new weighted $\ell^{p}$-space. Utilizing Theorem \ref{biforsuf}, we can demonstrate that the forward shift operator on this new weighted $\ell^{p}$-space exhibits distributional chaos.
\end{remark}

\section{Composition operators on $L^{p}(\mathbb{T},\mathcal{B},\lambda)$ induced by automorphisms of the complex unit disk $\mathbb{D}$}
Our last example is the composition operator $T_{\varphi}$ induced by an automorphism $\varphi$ of the complex unit disk $\mathbb{D}$, where $\mathcal{B}$ is the collection of all Lebesgue measurable subset of the unit circle $\mathbb{T}$ and $\lambda$ denotes the normalized Lebesgue measure on $\mathbb{T}$. Here, by automorphism we mean a biholomorphic self mapping on $\mathbb{D}$. It has been shown that $T_{\varphi}$ is topologically transitive if and only if $T_{\varphi}$ is topologically mixing if and only if $\varphi$ is an automorphism of the disk with no fixed point in $\mathbb{D}$ (Theorem 3.3 \cite{Bayart2018}). The $\mathcal{F}$-transitivity and d$\Fc$-transitivity were characterized for composition operators on $L^{p}(\mathbb{T},\mathcal{B},\lambda)$ induced by automorphisms on D in \cite{He2021}.

\begin{theorem}\label{lft}
Let $1\leq p<\infty$,  $\lambda$ denote the normalized Lebesgue measure on the unit circle $\mathbb{T}$, $\mathcal{B}$ be the collection of all Lebesgue measurable subset of  $\mathbb{T}$ and $T_{\varphi}$ be the composition operator induced by an automorphism $\varphi$ of $\mathbb{D}$ on $L^{p}(\mathbb{T},\mathcal{B},\lambda)$. Then the following assertions are equivalent:
\begin{enumerate}[label=\rm{(\roman*)}]
\item $T_{\varphi}$ is Li-Yorke chaotic;
\item $T_{\varphi}$ is distributionally chaotic;
\item $T_{\varphi}$ is densely distributionally chaotic;
\item $T_{\varphi}$ is topologically transitive;
\item $T_{\varphi}$ is topologically mixing;
\item $\varphi$ has no fixed point in $\mathbb{D}$. 
\end{enumerate}
\end{theorem}
\begin{proof}
The equivalence $(iv)\Leftrightarrow (v) \Leftrightarrow (vi)$ is from Theorem 3.3 \cite{Bayart2018}. It is sufficient to show that $(i)\Rightarrow (vi) $ and $(vi) \Rightarrow (iii)$.

$(i)\Rightarrow (vi).$ %From Proposition 4.36 in \cite{Peris}, any automorphism $\varphi$ of the unit disk $\mathbb{D}$ has the form $\varphi(z)=b\frac{a-z}{1-\overline{a}z},|a|<1,|b|=1$ and maps $\mathbb{T}$ bijectively onto itself. 
Since $T_{\varphi}$ is Li-Yorke chaotic, $\varphi$ cannot be the identity. Then from Proposition 4.47 \cite{Peris}, $\varphi$ must be in one of the following three cases.
\begin{enumerate}
\item $\varphi$ is parabolic. In this case, $\varphi$ has a single fixed point $z_{0}$ lying in $\mathbb{T}$, and $\varphi^{n}(z)\rightarrow z_{0}, \varphi^{-n}(z) \rightarrow z_{0}$ for all $z\in \widehat{\mathbb{C}}$.
\item $\varphi$ is elliptic. In this case, $\varphi$ has two distinct fixed points and one of which lies in $\mathbb{D}$.
\item $\varphi$ is hyperbolic. In this case, $\varphi$ has two distinct fixed points $z_{0}$ and $z_{1}$ in $\mathbb{T}$ such that $\varphi^{n}(z)\rightarrow z_{0}$ for all $z\in \widehat{\mathbb{C}}, z\neq z_{1}$ and $\varphi^{-n}(z)\rightarrow z_{1}$ for all $z\in \widehat{\mathbb{C}}, z\neq z_{0}$. 
\end{enumerate}

We only need to show that $\varphi$ cannot be elliptic. Otherwise, if $\varphi$ is elliptic, from Proposition 4.47 \cite{Peris}, there is a linear fractional transformation $\sigma$ and a unimodular complex constant $c\neq 1$ such that $\sigma(\mathbb{D})$ is a disk and $\sigma\circ \varphi\circ \sigma^{-1}(z)=cz, \forall z\in \widehat{\mathbb{C}}$. Let $\psi= \sigma\circ \varphi\circ \sigma^{-1}$, $\mathcal{B}'$ denote the collection of all Lebesgue measurable subset of the circle $\partial\sigma(\mathbb{D})$ and $\mu$ denote the normalized Lebesgue measure on the circle $\partial\sigma(\mathbb{D})$. It is not hard to check that $T_{\varphi}$ on $L^{p}(\mathbb{T},\mathcal{B},\lambda)$ is linearly topologically conjugate to the composition operator $T_{\psi}$ induced by $\psi$ on $L^{p}(\partial\sigma(\mathbb{D}),\mathcal{B}^{'},d\mu)$. Noting that $\psi$ is a rotation on the circle $\partial\sigma(\mathbb{D})$, by the characterization of Li-Yorke chaos given in Theorem 1.1 \cite{BernardesDarji}, $T_{\psi}$ is not Li-Yorke chaotic. This contradicts with that $T_{\varphi}$ is Li-Yorke chaotic.

$(vi)\Rightarrow (iii).$ The assumption that $\varphi$ has no fixed point in $\mathbb{D}$ is equivalent to that $\varphi$ is parabolic or hyperbolic.

Let $\alpha$ be the attractive fixed point and $\beta$ be the repulsive fixed point of $\varphi$ for each $1\leq j\leq m$ ($\alpha=\beta$ if $\varphi$ is parabolic). Take arbitrary $k\in \Nb$. Let $C_{k}=\mathbb{T}\backslash  I\left(\beta,\frac{1}{2k}\right),$ where $I\left(\beta,\frac{1}{2k}\right)$ denote the arc of $\mathbb{T}$ with measure $\frac{1}{2k}$ and midpoint $\beta$. It is clear that $\mu(\mathbb{T}\backslash \cup_{i\in \Nb}C_{i})=0$. Now we will show that $(C_{i})_{i\in I}$ satisfies the assumption in Proposition \ref{dcpropdense}. 

From \cite[Proposition 4.47]{Peris}, $\varphi^{n}\xrightarrow[n\rightarrow \infty]{}\alpha$ locally uniformly on $\mathbb{T}\backslash \beta$, then $\varphi^{n}\xrightarrow[n\rightarrow\infty]{}\alpha$ uniformly on $C_{i}$ for each $i\in \Nb.$ Hence $\lim\limits_{k\rightarrow\infty}\lambda\left(\varphi^{-k}(C_{i})\right)=0$ for any $i\in \Nb$ and $\varphi$ satisfies the assumption in Proposition \ref{dcpropdense}. Then the composition operator $T_{\varphi}$ is densely distributionally chaotic if and only if it is distributionally chaotic if and only if it is distributionally unbounded. Next we show that the compositon operator $T_{\varphi}$ is distributionally unbounded.

If $\varphi$ is hyperbolic, from Proposition 4.47 \cite{Peris}, there is a linear fractional transformation $\sigma:z\in\mathbb{C}\mapsto \frac{\alpha-z}{\beta-z}\in\widehat{\mathbb{C}}$ and a real number $0<\lambda<1$ such that $$\sigma\circ \varphi\circ \sigma^{-1}(z)=\lambda z, \forall z\in \widehat{\mathbb{C}}.$$ 
 
Then $\psi= \sigma\circ \varphi\circ \sigma^{-1}$ has the attractive fixed point $0$ and the repulsive fixed point $\infty$. Since $\sigma(\alpha) = 0$, $\sigma(\beta) = \infty$ and $\sigma$ is conformal as a linear fractional transformation, it follows that $\sigma(\mathbb{T})$ is a straight line passing through the origin, and $\sigma(\mathbb{D})$ is a half-plane. 

It is not difficult to compute that $\varphi'(\alpha)=(\sigma^{-1})'(0)\cdot\lambda\cdot\sigma'(\alpha)=\lambda$ and $(\varphi^{-1})'(\alpha)=1/\lambda.$
Since $(\varphi^{-1})'(z)$ is continuous on the whole complex plane, including $\alpha$, there exists some real number $\epsilon_{0}>0$, such that  $$|(\varphi^{-1})'(z)|>(1+|\lambda|)/2|\lambda|>1,$$ 
for any $z\in I(\alpha,\epsilon_{0})$. Since $\varphi^{-k}(\alpha)=\alpha,$ for any $k\in \Nb,$ then for any given $n\in \Nb,$ there exists some positive real number $\delta_{n}<\epsilon_{0}$ such that $\varphi^{-k}(I(\alpha,\delta_{n}))\subset I(\alpha,\epsilon_{0}),$ for any $0\leq k\leq n.$ Denote $M=(1+|\lambda|)/2|\lambda|>1$. Then $|\varphi^{-k})'(z)|>M>1$, for any $z\in I(\alpha,\delta_{n})$ and any $0\leq k\leq n$. It is clear that $\mu(\varphi^{-k}(I(\alpha,\delta_{n})))>M^{k}\mu(I(\alpha,\delta_{n})),$ for any $0\leq k\leq n$. Denote $B_{n}=I(\alpha,\delta_{n}),\ n\in \Nb$. Then we have that 
$$\left(\frac{\mu\left(\varphi^{-k}\left(B_{n}\right)\right)}{\mu\left(B_{n}\right)}\right)^{1/p}=\frac{M^{1/p}(1-M^{n/p})}{1-M^{1/p}}$$ 
for any $0\leq k\leq n, n\in \Nb$,
which implies that $T_{\varphi}$ is not distributionally unbounded.

If $\varphi$ is parabolic, $\varphi$ has a single fixed point $\alpha$ lying in $\mathbb{T}$, we can compute that the derivative of $\varphi'(\alpha)= 1$, which means that the proof method for hyperbolic types is no longer applicable. 

From Proposition 4.47 \cite{Peris}, there is a linear fractional transformation $\sigma:z\in\mathbb{C}\mapsto \frac{\alpha+z}{\alpha-z}\in\widehat{\mathbb{C}}$ and a real number $b\neq 0$ such that $$\psi:=\sigma\circ \varphi\circ \sigma^{-1}(z)=z+bi,\  \forall z\in \widehat{\mathbb{C}}.$$ Then $\psi= \sigma\circ \varphi\circ \sigma^{-1}$ has the attractive fixed point $\infty$. 

Denote $\alpha=e^{\theta_{\alpha}}$ where $\theta_{\alpha} \in [0,2\pi)$ and Denote $\theta(z)$ as the argument function of $z$, i.e. $z=|z|e^{\theta(z)}$, with the requirement that $\theta(z) \in [\theta_{\alpha},\theta_{\alpha}+2\pi)$. Denote the real part of the complex number $z$ as $\text{Re}(z)$ and the imaginary part as $\text{Im}(z)$. Next we give a simple but important computation:
\begin{equation}\label{sigma}
\sigma(z)=\frac{\alpha+z}{\alpha-z}=\frac{(\alpha+z)(\overline{\alpha}-\overline{z})}{(\alpha-z)(\overline{\alpha}-\overline{z})}=\frac{\text{Im} (\overline{\alpha} z) }{1-\text{Re} (\overline{\alpha} z)}i=\frac{\sin(\theta(z) - \theta(\alpha))}{1-\cos(\theta(z) - \theta(\alpha))}i,
\end{equation}
for any $z\in\mathbb{T}.$
From equation (\ref{sigma}), we have the following observations.

Observation 1: $\sigma(\mathbb{T})=\{\infty\}\cup i\mathbb{R}$,$\sigma(\alpha)=\infty$, $\sigma(-\alpha)={0}$. And $\sigma(z)={i}$ if $\theta(z)=\pi/2+\theta(\alpha).$

Observation 2: $\text{Im}(\sigma(z))\rightarrow  +\infty $ when $\theta(z) \rightarrow  \theta(\alpha)^{+}$ and  $\text{Im}(\sigma(z))\rightarrow  -\infty$ when $\theta(z) \rightarrow  (2\pi+\theta(\alpha))^{-}.$

Observation 3: $\theta(\sigma^{-1}(z))\rightarrow  \theta(\alpha)^{+} $ when $\text{Im}(z)\rightarrow  +\infty $ and  $\theta(\sigma^{-1}(z))\rightarrow  (2\pi+\theta(\alpha))^{-}$ when $\text{Im}(z)\rightarrow  -\infty.$

Observation 4: $\text{Im}(\sigma(z))<\text{Im}(\sigma(z'))$ if $\theta(z)>\theta(z')$ and $z'\neq \alpha.$ $\theta(z)>\theta(z')$ if $\text{Im}(\sigma(z))<\text{Im}(\sigma(z')).$ 

Note that $\psi^{n}(z)=z+nbi$ and $\psi^{-n}(z)=z-nbi$, for any $n\in \Nb$ and $z\in \widehat{\mathbb{C}}$. If $b>0,$  
for any given $z\in \mathbb{T},z\neq \alpha,$ as $n\rightarrow +\infty$, we have that  $$\text{Im} (\psi^{n}(\sigma(z)))\xrightarrow[\text{strictly increasing}]{}+\infty, \ \text{Im} (\psi^{-n}(\sigma(z)))\xrightarrow[\text{strictly decreasing}]{} -\infty,$$
which implies 
\begin{equation}\label{infty}
\theta(\sigma^{-1}(\psi^{n}(\sigma(z))))\xrightarrow[\text{strictly decreasing}]{}  \theta(\alpha)^{+},\ \theta(\sigma^{-1}(\psi^{-n}(\sigma(z))))\xrightarrow[\text{strictly increasing}]{}  (2\pi+\theta(\alpha))^{-},
\end{equation}
from observation 3 and observation 4.

Noting that $\sigma^{-1}\circ \psi^{n}\circ \sigma=\varphi^{n}$ and $\sigma^{-1}\circ \psi^{-n}\circ \sigma=\varphi^{-n}$ for any $n\in \Nb$ and $z\in \widehat{\mathbb{C}}$. From formula (\ref{infty}), we can see that as $n\rightarrow +\infty,$
\begin{equation}\label{monoinf}
\theta(\varphi^{n}(z))\xrightarrow[\text{strictly decreasing}]{}  \theta(\alpha)^{+},\ \theta(\varphi^{-n}(z))\xrightarrow[\text{strictly increasing}]{}  (2\pi+\theta(\alpha))^{-},
\end{equation}

Let $y=\varphi^{-1}(z)=\sigma^{-1}\circ\psi^{-1}\circ\sigma(z)$. In the following, our main goal is to estimate the derivative $\frac{dy}{dz}$.
Noting that $\sigma(y)=\psi^{-1}\cdot\sigma(z)$, from equation (\ref{sigma}), we have that
\begin{equation}
\frac{\sin(\theta(y)-\theta(\alpha)) i}{1-\cos(\theta(y)-\theta(\alpha))}=\frac{\sin(\theta(z)-\theta(\alpha)) i}{1-\cos(\theta(z)-\theta(\alpha))}-bi.
\end{equation}

Then we have that 
\begin{equation}\label{inverse}
F(z,y):=\frac{\sin(\theta(y)-\theta(\alpha)) }{1-\cos(\theta(y)-\theta(\alpha))}-\frac{\sin(\theta(z)-\theta(\alpha)) i}{1-\cos(\theta(z)-\theta(\alpha))}+bi=0.
\end{equation}
Thus, we can obtain the derivative
\begin{equation}\label{derivative}
\frac{dy}{dz}=\frac{d\theta(y)}{d\theta(z)}=-\frac{F_{\theta(z)}}{F_{\theta(y)}}=\frac{1-\cos(\theta(y) - \theta(\alpha))}{1-\cos(\theta(z) - \theta(\alpha))}.
\end{equation}
Suppose $\text{Im}(\sigma(z))\geq \left(1+\left\lceil \frac{1}{b} \right\rceil \right)b\geq b+1$, i.e., $\text{Im}(\sigma(y))\geq 1$. Note that $\text{Im}(\sigma(e^{\theta(\alpha)+\pi/2}))=1$. From the monotonicity in the latter part of formula (\ref{monoinf}), we can deduce that $\theta(y) > \theta(z)$. Then we have that $0\leq\theta(z) - \theta(\alpha)\leq\theta(y) - \theta(\alpha)<\pi/2$, which implies that $\sin(\theta(y) - \theta(\alpha))>\sin(\theta(z) - \theta(\alpha))>0.$ Hence we have that 
\begin{equation}\label{derivative2}
\frac{dy}{dz}=\frac{\frac{1}{1-\cos(\theta(z) - \theta(\alpha))}}{\frac{1}{1-\cos(\theta(y) - \theta(\alpha))}}>\frac{\frac{\sin(\theta(z) - \theta(\alpha))}{1-\cos(\theta(z) - \theta(\alpha))}}{\frac{\sin(\theta(y) - \theta(\alpha))}{1-\cos(\theta(y) - \theta(\alpha))}}=\frac{\frac{\sin(\theta(z) - \theta(\alpha))}{1-\cos(\theta(z) - \theta(\alpha))}}{\frac{\sin(\theta(z) - \theta(\alpha))}{1-\cos(\theta(z) - \theta(\alpha))}-b}=\frac{\text{Im}(\sigma(z))}{\text{Im}(\sigma(z))-b}.
\end{equation}

For any $j\geq 2+\left\lceil \frac{1}{b} \right\rceil$, denote by $T_{j}\subset \mathbb{T}$ the arc such that 
\begin{equation}\label{tj}
\text{Im}(\sigma(T_{j}))= [(j-1)b,jb).
\end{equation}
Let $B_{j}\subset T_{j}$ be any measurable subset of $T_{j}.$

Then we have that 
$$\text{Im}(\sigma^{-k}(B_{j}))\subset [(j-1-k)b,(j-k)b)$$
 for any $1\leq k\leq j-1-\left\lceil \frac{1}{b} \right\rceil.$ 
From equatin (\ref{derivative2}), we have that 
\begin{equation}
\left.\frac{dy}{dz}\right|_{z=z_{0}}=\frac{\text{Im}(\sigma(z_{0}))}{\text{Im}(\sigma(z_{0}))-b}\in \left[\frac{j-k}{j-k-1},\frac{j-k-1}{j-k-2}\right],
\end{equation}
for any $z_{0}\in\sigma^{-k}(B_{j}),1\leq k\leq j-2-\left\lceil \frac{1}{b} \right\rceil.$

Hence we have that 
\begin{equation}\label{uk}
\mu(\varphi^{-k}(B_{j}))=\int_{\varphi^{-k}(B_{j})}dy\geq \int_{\varphi^{-(k-1)}(B_{j})}\frac{j-(k-1)}{j-(k-1)-1}dz=\frac{j-k+1}{j-k}\mu(\varphi^{-(k-1)}(B_{j})),
\end{equation}
and that
\begin{equation}\label{uk}
\mu(\varphi^{-k}(B_{j}))\geq \frac{j}{j-k}\mu(B_{j}),
\end{equation}
for any $1\leq k\leq j-1-\left\lceil \frac{1}{b} \right\rceil.$ 

In particular, when $k=1$ and $k= j-1-\left\lceil \frac{1}{b} \right\rceil,$ we have that 
\begin{equation}
\mu(\varphi^{-1}(B_{j}))\geq\frac{j}{j-1}\mu(B_{j}) \text{ and }
\mu(\varphi^{-(j-1-\left\lceil \frac{1}{b} \right\rceil)}(B_{j}))\geq \frac{2+\left\lceil \frac{1}{b} \right\rceil}{1+\left\lceil \frac{1}{b} \right\rceil}\mu(B_{j}).
\end{equation}

For any given $i\in\Nb,$ let $n_{i}=2^{i},$ and take arcs $B_{j}\subset T_{j}$ with measure $\mu(B_{j})=M_{i}=\mu\left(T_{in_{i}+2+\left\lceil \frac{1}{b} \right\rceil}\right),$ for any $,j=1,2,\cdots,in_{i}+2+\left\lceil \frac{1}{b} \right\rceil.$ From equation (\ref{tj}), we have that 
$\varphi^{-k}(B_{i})\cap \varphi^{-k}(B_{i'})=\emptyset$ for any distinct pair $i,i' \in \Nb$ and any $k\in \Nb.$ Take a simple function $f_{i}=\sum_{j=n_{i}}^{in_{i}+2+\left\lceil \frac{1}{b} \right\rceil}\frac{1}{j^{1/p}}\mathcal{X}_{B_{j}}$ with that 
\begin{equation}
\Vert f_{i} \Vert^{p}= \sum_{j=n_{i}}^{in_{i}}\frac{1}{j}\mu(B_{j})=\mu\left(T_{in_{i}+2+\left\lceil \frac{1}{b} \right\rceil}\right)(\ln (k)).
\end{equation}

For any $k\in [n_{i},(i-1)n_{i})$, we have that 
\begin{align}
\Vert T_{\varphi}^{k}f_{i} \Vert^{p}&= \sum_{j=n_{i}}^{in_{i}}\frac{1}{j}\mu(\mathcal{X}_{\varphi^{-k}(B_{j})})\geq \sum_{j=k+1+\left\lceil \frac{1}{b} \right\rceil}^{in_{i}}\frac{1}{j}\cdot\frac{j}{j-k}\mu(B_{j})\\
& =\sum_{j'=1+\left\lceil \frac{1}{b} \right\rceil}^{in_{i}-k}\frac{1}{j'}\delta_{i}\geq \left(\ln(in_{i}-k)-\ln(1+\left\lceil \frac{1}{b} \right\rceil)\right)\mu\left(T_{in_{i}+2+\left\lceil \frac{1}{b} \right\rceil}\right)\\
&\geq \left(\ln (n_{k})-\ln(1+\left\lceil \frac{1}{b} \right\rceil)\right)\mu\left(T_{in_{i}+2+\left\lceil \frac{1}{b} \right\rceil}\right). \tag*{}
\end{align}
Therefore,  condition (ii) in Theorem \ref{dcsufandnec} holds and $T_{\varphi}$ is distributionally unbounded. Now we conclude that $T_{\varphi}$ is densely distributionally chaotic from Proposition \ref{dcpropdense}.

\end{proof}

\subsection*{Acknowledgements} This research was funded by the National Natural Science Foundation of China (No.  12101415, 62272313), Shenzhen Institute of Information Technology (No. SZIIT2022KJ008), Science and Technology Projects in Guangzhou (No. 2024A04J4429) and the project of promoting research capabilities for key constructed disciplines in Guangdong Province (No. 2021ZDJS028).

\end{document}